\documentclass[12pt]{amsart}

\usepackage{amsmath}
\usepackage{amscd}
\usepackage{amssymb}
\usepackage{amsfonts}
\usepackage{amsthm}
\usepackage{enumerate}
\usepackage{hyperref}
\usepackage{color}
\usepackage{graphicx}
\theoremstyle{plain}
\newtheorem{thm}{Theorem}[section]

\newtheorem{prop}[thm]{Proposition}
\newtheorem{lem}[thm]{Lemma}

\newtheorem{ques}[thm]{Question}

\theoremstyle{definition}

\newtheorem{dfns-rems}[thm]{Definitions and Remarks}
\newtheorem{notas-rems}[thm]{Notations and Remarks}
\newtheorem{exmps-rems}[thm]{Examples and Remarks}

\begin{document}

% ------------------------------------------------------------------------

\title[$h$-vector of simplicial complexes]{On the $h$-vector of ($S_r$) simplicial complexes}

% ------------------------------------------------------------------------

\author[S. A. Seyed Fakhari]{S. A. Seyed Fakhari}

\address{S. A. Seyed Fakhari, School of Mathematics, Statistics and Computer Science,
College of Science, University of Tehran, Tehran, Iran.}

\email{aminfakhari@ut.ac.ir}

\urladdr{http://math.ipm.ac.ir/$\sim$fakhari/}

% ------------------------------------------------------------------------

\begin{abstract}
We give a negative answer to a question proposed in \cite{gpsy}, regarding the $h$-vector of ($S_r$) simplicial complexes.
\end{abstract}

% ------------------------------------------------------------------------


% ------------------------------------------------------------------------


% ------------------------------------------------------------------------

\maketitle

%%%%%%%%%%%%%%%%%%%%%%%%%%%%%%%%%%%%%%%%%%%%%%%%%%%%%%%%%%%%%%%%%%%%%%%%%

\section{Introduction and preliminaries} \label{sec1}

The study of $h$-vectors of simplicial complexes is an important topic in
combinatorial commutative algebra, because it determines the coefficients
of the numerator of the Hilbert series of a Stanley--Reisner ring associated
to a simplicial complex. We refer the reader to Stanley's book \cite{s'} and the book of Herzog and Hibi \cite{hh} for an introduction to Simplicial complexes and Stanley--Reisner rings.

Let $\mathbb{K}$ be a field and $S=\mathbb{K}[x_1,\dots,x_n]$ be the
polynomial ring in $n$ variables over $\mathbb{K}$. A finitely generated $S$-module $M$ is said to satisfy the {\it Serre's condition} ($S_r$), if $${\rm depth}\ M_{\frak{p}}\geq \min\{r,\dim
M_{\frak{p}}\},$$ for every $\frak{p}\in
\rm{Spec}(S)$. We say that a
simplicial complex $\Delta$ is an ($S_r$) {\it simplicial complex}, if its
Stanley--Reisner ring satisfies the Serre's condition ($S_r$). It is easy to see that every
simplicial complex is ($S_1$). Therefore, we assume that
$r\geq2$. We refer the reader to \cite{psty} for a survey about ($S_r$) simplicial complexes.

The classical result of Stanley
characterizes the $h$-vector of Cohen--Macaulay simplicial complexes (see \cite[Theorem 3.3, Page 59]{s'}). Murai and Terai \cite{mt} studied the $h$-vector of ($S_r$) simplicial complexes. They proved that if $\Delta$ is a $(d-1)$-dimensional
($S_r$) simplicial complex with $h(\Delta)=(h_0, \ldots, h_d)$, then $(h_0,h_1,\ldots,h_r)$ is an $M$-vector (i.e., it is the $h$-vector of a Cohen--Macaulay simplicial complex) and $h_r+h_{r+1}+
\cdots+h_d$ is nonnegative. In \cite{gpsy}, the authors extended the result of Murai and Terai
by giving $r$ extra necessary conditions. Indeed, they proved that$${i\choose i}h_r+{i+1\choose i}h_{r+1}+\cdots+{i+d-r\choose
    i}h_d\geq 0,$$ for every integer $i$ with $1\leq i\leq r$. Notice that for $i=0$, the above inequality reduces to the inequality
$h_r+ \ldots+ h_d\geq 0$, which was obtained by Murai and Terai. In \cite{gpsy}, the authors asked whether the above mentioned conditions are also sufficient for a sequence of integers to be the $h$-vector of a ($S_r$) simplicial complex. In fact, they proposed the following question.

\begin{ques}[\cite{gpsy}, Question 2.6] \label{quest}
Let $d$ and $r$ be integers with $d\geq r\geq2$ and let $\mathbf{h}=(h_0,
h_1,\ldots,h_d)$ be the $h$-vector of a simplicial complex in such a way
that the following conditions hold:

\vspace{0.2cm}

\begin{itemize}
\item[(1)] $(h_0,h_1,\ldots,h_r)$ is an $M$-vector, and \vspace{0.3cm}

\item[(2)] ${i\choose i}h_r+{i+1\choose i}h_{r+1}+\cdots+{i+d-r\choose
    i}h_d$ is nonnegative for every $i$ with $0\leq i\leq r$.
\end{itemize}

\vspace{0.2cm}

\noindent Does there exist a $(d-1)$-dimensional ($S_r$) simplicial complex
$\Delta$ with $h(\Delta)=\mathbf{h}$?
\end{ques}

In this paper, we give a negative answer to this question, by presenting a class of infinitely many sequences which satisfy the assumptions of Question \ref{quest} for $r=2$, while they are not the $h$-vector of any ($S_2$) simplicial complex. It is still interesting to know whether Question \ref{quest} is true in the case $r\geq 3$.

Another result obtained by Murai and Terai \cite[Theorem 1.2]{mt} states that if $\mathbf{h}=(h_0,h_1,\ldots,h_d)$ is the $h$-vector of a ($S_r$) simplicial complex and $h_i=0$ for some
$i\leq r$, then $h_k=0$ for all $k\geq i$. This is in fact a necessary condition for a sequence of integers to be the $h$-vector of a ($S_r$) simplicial complex. Our example shows that if we add this necessary condition to the assumptions of Question \ref{quest}, then the answer would be still negative.

Let $\mathbf{h}=(h_0, h_1,\ldots,h_d)$ be a sequence of integers. One may ask, whether there exists a ($S_2$) simplicial complex
$\Delta$ with $h(\Delta)=\mathbf{h}$, provided that $\mathbf{h}$ satisfies the conditions (1) and (2) of Question \ref{quest} and moreover, $\mathbf{h}$ is the $h$-vector of a "pure" simplicial complex. In fact, Our example shows that the answer of this question is also negative (see Lemma \ref{pure}). More explicit, we prove the following result.

\begin{thm} \label{main}
For every integer $d\geq 5$, there exists a vector $\mathbf{h}=(h_0,
h_1,\ldots,h_d)$ of nonzero integers which is the $h$-vector of a pure simplicial complex and moreover,

\vspace{0.2cm}

\begin{itemize}
\item[(1)] $(h_0,h_1,h_2)$ is an $M$-vector, and \vspace{0.3cm}

\item[(2)] ${i\choose i}h_2+{i+1\choose i}h_3+\cdots+{i+d-2\choose
    i}h_d$ is nonnegative for every $i$ with $0\leq i\leq 2$.
\end{itemize}

\vspace{0.2cm}

\noindent But there is no $(d-1)$-dimensional ($S_2$) simplicial complex
$\Delta$ with $h(\Delta)=\mathbf{h}$.
\end{thm}

%%%%%%%%%%%%%%%%%%%%%%%%%%%%%%%%%%%%%%%%%%%%%%%%%%%%%%%%%%%%%%%%%%%%%%%%%%

\section{Proof of Theorem \ref{main}} \label{sec2}

Let $d\geq 5$ be an integer and set $\mathbf{h}=(1,2, \underbrace{1, \ldots, 1,}_{(d-2)\text{ times}} -1)$. In Lemma \ref{pure}, we prove that $\mathbf{h}$ is the $h$-vector of a pure simplicial complex. Before stating this lemma, we remind that for a graded $S$-module $M=\oplus_{i\geq 0}M_i$, the {\it Hilbert series} of $M$ is defined to be$${\rm Hilb}_M(t)=\sum_{i\geq 0}({\rm dim}_{\mathbb{K}}M_i)t^i.$$It is well-known that for a $(d-1)$-dimensional simplicial complex $\Delta$, with $h(\Delta)=(h_0,
h_1,\ldots,h_d)$, we have$${\rm Hilb}_{\mathbb{K}[\Delta]}(t)=\frac{h_0+h_1t+h_2t^2+ \ldots +h_dt^d}{(1-t)^d}.$$

\begin{lem} \label{pure}
Assume that $d\geq 5$ is an integer. Let $\Delta$ be the simplicial complex over $[d+2]$ with facets$$\mathcal{F}(\Delta)=\Big\{[d+2]\setminus\{1,j\}: 2\leq j \leq d\Big\}\bigcup\Big\{[d+2]\setminus\{d+1,d+2\}\Big\}.$$Then $h(\Delta)=(1,2, \underbrace{1, \ldots, 1,}_{(d-2)\text{ times}} -1)$.
\end{lem}

\begin{proof}
Set $n=d+2$. By \cite[Lemma 1.5.4]{hh}, we have$$I_{\Delta}=\Big(\bigcap_{2\leq j\leq d}(x_1, x_j)\Big)\cap (x_{d+1}, x_{d+2}).$$Set $L=\bigcap_{2\leq j\leq d}(x_1, x_j)=(x_1, x_2x_3\ldots x_d)$ and $K=(x_{d+1}, x_{d+2})$. Then $I_{\Delta}=L\cap K$. Consider the following exact sequence of graded $S$-modules:
\[
\begin{array}{rl}
0\longrightarrow S/I_{\Delta}\longrightarrow S/L\oplus S/K\longrightarrow S/(L+K)
\longrightarrow 0.
\end{array}
\]
It follows that$${\rm Hilb}_{S/I_{\Delta}}(t)={\rm Hilb}_{S/L}(t)+{\rm Hilb}_{S/K}(t)-{\rm Hilb}_{S/L+K}(t).$$Notice that$${\rm Hilb}_{S/L}(t)={\rm Hilb}_{\mathbb{K}[x_2, \ldots, x_n]/(x_2x_3 \ldots x_d)}(t)=\frac{1-t^{d-1}}{(1-t)^{d+1}},$$where that last equality follows from the fact $x_2x_3 \ldots x_d$ is a regular element of $\mathbb{K}[x_2, \ldots, x_n]$ with degree $d-1$. Similarly,$${\rm Hilb}_{S/K}(t)=\frac{1}{(1-t)^d}$$and$${\rm Hilb}_{S/L+k}(t)=\frac{1-t^{d-1}}{(1-t)^{d-1}}.$$A simple computation, using the above equalities shows that $${\rm Hilb}_{S/I_{\Delta}}(t)=\frac{1+2t+t^2+t^3+ \ldots +t^{d-1}-t^d}{(1-t)^d}.$$Hence, $h(\Delta)=(1,2, \underbrace{1, \ldots, 1,}_{(d-2)\text{ times}} -1)$.
\end{proof}

It can be easily seen that $(1, 2, 1)$ is an $M$-vector. Indeed, it is the $h$-vector of the simplicial complex over $[4]$ with facets $\{1, 2\}, \{2, 3\}, \{3, 4\}$ and $\{4, 1\}$, which is Cohen--Macaulay. On the other hand, we are assuming that $d\geq 5$ and thus $$h_2+ h_3+ \ldots +h_d=d-3\geq 0,$$ $$h_2+2h_3+ \ldots + (d-1)h_d=1+2+ \ldots + (d-2)-(d-1)=\frac{(d-4)(d-1)}{2}\geq 0$$ and $$h_2+ 3h_3+ \ldots + {d \choose 2}h_d=1+3+ \ldots + {d-1 \choose 2}-{d \choose 2}={d\choose 3}-{d\choose 2}\geq 0,$$where the last equality follows from \cite[Page 368, Theorem 4]{r}. Thus, $\mathbf{h}$ satisfies the assumptions of Theorem \ref{main}. We show in Proposition \ref{s2} that $\mathbf{h}$ is not the $h$-vector of any ($S_2$) simplicial complex.

We first remind some definitions and basic facts.

Let $\Delta$ be a $(d-1)$-dimensional simplicial complex with vertex set $[n]$. Assume that $f(\Delta)=(f_0, \ldots, f_{d-1})$ and $h(\Delta)=(h_0, \ldots, h_d)$ are the $f$-vector and $h$-vector of $\Delta$, respectively. It is well-known (see e.g. \cite[Corollary 5.1.9]{bh}) that

\[
\begin{array}{rl}
h_1=n-d \ \ \ \ \ {\rm and} \ \ \ \ \ h_0+h_1+ \ldots +h_d=f_{d-1}
\end{array} \tag{$\ast$} \label{ast}
\]

A simplicial complex $\Delta$ is called a {\it cone} if it has a vertex which belongs to every facet of $\Delta$.

The proof of the following lemma is simple and is omitted.

\begin{lem} \label{cone}
Let $\Delta$ be a $(d-1)$-dimensional cone, with $h(\Delta)=(h_0, \ldots, h_d)$. Then $h_d=0$.
\end{lem}

Let $G$ be a graph with vertex set $V(G)=\big\{v_1, \ldots, v_n\big\}$ and edge set $E(G)$. The {\it complementary graph} $\overline{G}$ is a graph with $V(\overline{G})=V(G)$ and $E(\overline{G})$ consists of those $2$-element subsets $\{v_i,v_j\}$ of $V(G)$ for which
$\{v_i,v_j\}\notin E(G)$. A subset $C$ of $V(G)$ is called a {\it vertex cover} of the graph $G$ if every edge of $G$ is incident to at least one vertex of $C$. The {\it cover ideal} of $G$, denoted by $J(G)$, is the squarefree monomial ideal which is generated by the set$$\big\{\prod_{v_i\in C}x_i : C \ {\rm is \ a \ vertex \ cover \ of} \ G\big\}.$$By \cite{v}, we know tha for every graph $G$,$$J(G)=\bigcap_{\{v_i, v_j\}\in E(G)}(x_i, x_j).$$A monomial ideal is said to be {\it unmixed} if all its associated primes have the same height. The above equality shows that a squarefree monomial ideal is unmixed of height $2$ if and only it is the cover ideal of a graph.

We are now ready to complete the proof of Theorem \ref{main}.

\begin{prop} \label{s2}
Let $d\geq 5$ be an integer. Then there is no ($S_2$) simplicial complex $\Delta$ with $h(\Delta)=(1,2, \underbrace{1, \ldots, 1,}_{(d-2)\text{ times}} -1)$. In particular, the answer of Question \ref{quest} is in general negative.
\end{prop}

\begin{proof}
Assume by contradiction that there exists a ($S_2$) simplicial complex $\Delta$ with $h(\Delta)=(1,2, \underbrace{1, \ldots, 1,}_{(d-2)\text{ times}} -1)$. Thus, ${\rm dim}(\Delta)=d-1$ and it follows from Lemma \ref{cone} that $\Delta$ is not a cone. By \cite[Lemma 2.6]{mt}, we know that $\Delta$ is a pure simplicial complex. Therefore, the equalities \ref{ast} imply that $\Delta$ has $d+2$ vertices and $d$ facets.  It then follows from \cite[Lemma 1.5.4]{hh} that $I_{\Delta}$ is an unmixed ideal of height $2$. Hence, $I_{\Delta}$ is the cover ideal of a graph, say $G$. Since $\Delta$ is not a cone, $G$ has no isolated vertex. The number of edges of $G$ is equal to the number of facets of $\Delta$, which is $d$. On the other hand, $G$ has $d+2$ vertices. Thus, $G$ is not a connected graph.  Let $H_1$ and $H_2$ be two connected components of $G$. Assume that $\{u, v\}$ and $\{z, t\}$ are edges of $H_1$ and $H_2$, respectively. Then $u, z, v, t$ is an induced $4$-cycle in $\overline{G}$. On the other hand, $S/J(G)=S/I_{\Delta}$ satisfies the Serre's condition ($S_2$) and it follows from \cite[Corollary 3.7]{y} and \cite[Theorem 2.1]{eghp} (see also \cite[Theorem 5.8]{psty}) that $\overline{G}$ can not have any induced $4$-cycle, which is a contradiction. Therefore, $((1,2, \underbrace{1, \ldots, 1,}_{(d-2)\text{ times}} -1)$ is not the $h$-vector of any ($S_2$) simplicial complex.
\end{proof}

%%%%%%%%%%%%%%%%%%%%%%%%%%%%%%%%%%%%%%%%%%%%%%%%%%%%%%%%%%%%%%%%%%%%%%%%%%

\section*{Acknowledgment}

The author thanks Naoki Terai and Siamak Yassemi for reading an earlier version of the paper and their helpful comments. He also thanks the referee for careful reading of the paper and for valuable comments. This work was supported by a grant from Iran National Science Foundation: INSF (No. 95820482).

%%%%%%%%%%%%%%%%%%%%%%%%%%%%%%%%%%%%%%%%%%%%%%%%%%%%%%%%%%%%%%%%%%%%%%%%%%

%%%%%%%%%%%%%%%%%%%%%%%%%%%%%%%%%%%%%%%%%%%%%%%%%%%%%%%%%%%%%%%%%%%%%%%%%%

\end{document}